\documentclass[a4paper, 10pt]{amsart}
\usepackage{amsmath}
\usepackage{amssymb, color, hyperref}

\theoremstyle{plain}
\newtheorem{theorem}{\bf Theorem}[section]
\newtheorem{proposition}[theorem]{\bf Proposition}
\newtheorem{lemma}[theorem]{\bf Lemma}
\newtheorem{corollary}[theorem]{\bf Corollary}

\theoremstyle{definition}
\newtheorem{example}[theorem]{\bf Example}

\newtheorem{definition}[theorem]{\bf Definition}
\newtheorem{remark}[theorem]{\bf Remark}

\newcommand{\N}{\mathbb N}
\newcommand{\Z}{\mathbb Z}

\newcommand{\Q}{\mathbb Q}

\DeclareMathOperator{\spec}{spec}

\numberwithin{equation}{section}

\begin{document}

\title{A characterization of weakly Krull monoid algebras}

\author{Victor Fadinger and Daniel Windisch}

\address{University of Graz, NAWI Graz \\
Institute for Mathematics and Scientific Computing \\
Heinrichstra{\ss}e 36\\
8010 Graz, Austria}
\email{victor.fadinger@uni-graz.at}
\address{Graz University of Technology, NAWI Graz\\
Institute for Analysis and Number Theory\\
Kopernikusgasse 24/II\\
8010 Graz, Austria}
\email{dwindisch@math.tugraz.at}
\urladdr{}

\thanks{This work was supported by the Austrian Science Fund FWF, Projects W1230 and P~30934.}


\begin{abstract}
Let $D$ be a domain and let $S$ be a torsion-free monoid whose quotient group satisfies the ascending chain condition on cyclic subgroups. We give a characterization of when the monoid algebra $D[S]$ is weakly Krull. As corollaries, we obtain the results on when $D[S]$ is Krull resp. generalized Krull, due to Chouinard resp. El Baghdadi and Kim. Furthermore, we deduce Chang's theorem on weakly factorial monoid algebras and we characterize the weakly Krull domains among the affine monoid algebras. 
\end{abstract}

\maketitle

\section{Introduction}

When considering monoid algebras $D[S]$, an immediate question is whether certain properties of the domain $D$ and the monoid $S$ carry over to the monoid algebra and conversely. A lot of such properties are studied in the textbook by Gilmer \cite{Gil84}, among them the property of being a Krull domain. For monoid algebras, this property is completely characterized in terms of the domain and the monoid, and this characterization was originally proved by Chouinard \cite{Chou81}. There he was the first to give the definition of a Krull monoid, which in turn was preparing the ground for a whole new research area based on the fact that a domain is Krull if and only if its multiplicative monoid is a Krull monoid. For a detailed study of Krull monoids, see \cite{Ge-HK06a}.\\
Many well-studied classes of domains (e.g. non-principal orders in number fields) fail to be Krull, but still share many important properties with Krull domains. To investigate them, two generalizations of Krull domains and monoids were introduced; namely generalized Krull and weakly Krull domains and monoids. The latter were first studied in \cite{AMZ92}, although not named weakly Krull domains there. It still holds true, that a domain is generalized (resp. weakly) Krull if and only if its multiplicative monoid is generalized (resp. weakly) Krull. A profound introduction to these classes of monoids can be found in \cite{HK98}. A divisor-theoretic characterization of weakly Krull monoids was first given by Halter-Koch in \cite{HK95}, where one can also find a proof of the statement that a domain is weakly Krull if and only if its multiplicative monoid is.\\
In 2009, Chang \cite{Chang09} gave a characterization of when the monoid algebra is weakly factorial (i.e. every non-zero non-unit is a product of primary elements) under the assumption that the quotient group of the monoid satisfies the ACC (ascending chain condition) on cylic subgroups. Namely, $D[S]$ is weakly factorial if and only if $D$ is a weakly factorial GCD-domain and $S$ is a weakly factorial GCD-monoid. In this paper, he also posed the question of when a monoid algebra is weakly Krull (note that a domain is weakly factorial if and only if it is  weakly Krull with trivial $t$-class group). In 2016, El Baghdadi and Kim \cite{BK2016} gave a complete characterization of when a monoid algebra is generalized Krull (namely if and only if both, the domain and the monoid are generalized Krull and the quotient group of the monoid satisfies the ACC on cyclic subgroups). An immediate consequence of the results by Chouinard, and El Baghdadi and Kim is the fact that the polynomial ring over a (generalized) Krull domain is (generalized) Krull and one would hope so as well for the case of weakly Krull domains. That this is not the case was proven by D.D. Anderson, Houston and Zafrullah in \cite[Proposition 4.11]{AHZ93}. In fact, they showed that this is the case precisely for weakly Krull UMT-domains. Recall that a domain is said to be a UMT-domain if all prime ideals of the polynomial ring $D[X]$ lying over $(0)$ are maximal $t$-ideals.\\
Answering Chang's question and trying to close the last gap for a complete picture of the generalizations of Krull domains, in Section 3, we give a characterization of when a monoid algebra is weakly Krull under the assumption that the quotient group of the monoid satisfies the ACC on cyclic subgroups. As a proof of the applicability of our description, in Section 4, we characterize the weakly Krull domains among the affine monoid algebras, which play an important role in the study of polytopes (see \cite{BG09}). Moreover, we reobtain the results by Chouinard, El Baghdadi and Kim, and Chang.


\section{Preliminaries}

We assume some familiarity with monoid algebras and ideal systems. We consider monoids to be cancellative, commutative and unitary semigroups and write them additively. For a domain $D$ and a monoid $S$, let $D[S]$ denote the monoid algebra of $D$ over $S$. It is well known that the monoid algebra $D[S]$ is a domain if and only if the monoid $S$ is torsion-free \cite[Theorem 8.1]{Gil84} and in that case $S$ amits a total order $<$ compatible with its semigroup operation. Therefore we can write every element $f\in D[S]$ in the form $f=\sum_{i=1}^{n}d_iX^{s_i}$ with $d_i\in D$, $s_i\in S$ and $s_1<s_2<\hdots <s_n$. For subsets $P\subseteq D$ and $H\subseteq S$, we denote by $P[H]=\{\sum_{i=1}^{n}p_iX^{h_i}\mid n\in\N, p_i\in P, h_i\in H\}\subseteq D[S]$.\\
A group satisfies the ACC (ascending chain condition) on cyclic subgroups if and only if it is of type $(0,0,\hdots )$. For a proof of this statement and for further equivalent conditions, see \cite[\S 14]{Gil84}.\\
Let $S$ be a monoid. We denote by
\begin{itemize}
\item $\mathsf q(S)$ the \textit{quotient group} of $S$,
\item $\widetilde{S}=\{x\in\mathsf q(S)\mid nx\in S\text{ for some } n\in\N\}$ the \textit{root closure} of $S$,
\item $\widehat{S}=\{x\in\mathsf q(S)\mid \text{there is } d\in S \text{ s.t. } d+nx\in S\text{ for all } n\in\N \}$ the \textit{complete integral closure} of $S$.
\end{itemize}
Let $\mathsf q(S)=G$ and let $X\subseteq S$ be a subset. We set $X^{-1}=(S:_G X)=\{g\in G\mid g+X\subseteq S\}$, $X_v=(X^{-1})^{-1}$ and $X_t=\bigcup_{J\subseteq X, |J|<\infty}J_v$.  We say that an ideal $I$ of $S$ is a $v$-\textit{ideal} (resp. \textit{t-ideal}) if $I=I_v$ (resp. $I=I_t$).
$I$ is called a \textit{maximal} $t$-ideal, if $I$ is maximal among all proper $t$-ideals of $S$ (and is therefore necessarily prime). $S$ is said to be of $t$-\textit{dimension} 1 ($t$-$\dim(S)=1$) if each maximal $t$-ideal of $S$ is a minimal non-empty prime ideal. An analogous concept exists for domains and the reader is referred to \cite{Gil92} for the domain case and to \cite{HK98} for the monoid case. We denote by $\mathfrak X(D)$ (resp. $\mathfrak X(S)$) the height-one spectrum of $D$ (resp. the minimal non-empty prime ideals of $S$) and will often also call elements of $\mathfrak X(S)$ height-one prime ideals. A domain (resp. monoid) $D$ is called \textit{weakly Krull} if $D=\bigcap_{P\in\mathfrak X(D)}D_P$ and $\mathfrak X(D)$ is of finite character, meaning that every non-zero non-unit (resp. non-unit) is contained only in a finite number of height-one prime ideals (resp. minimal non-empty prime ideals). A weakly Krull domain (resp. monoid) $D$ is called \textit{generalized Krull} if $D_P$ is a valuation domain (resp. monoid) for all $P\in\mathfrak X(D)$. For examples of generalized Krull monoids that do not stem from domains, we refer to \cite{BGGS} and \cite{CHKK02}.


\section{Main Result}

We start with an investigation of the height-one spectrum of monoid algebras, which will lead us to the definition of the central property characterizing weakly Krull monoid algebras.

\begin{lemma}\label{2.1}
Let $D$ be a domain with quotient field $K$ and let $S$ be a torsion-free monoid with quotient group $G$. For the height-one spectrum of $D[S]$ the following hold:
\begin{equation*}
\mathfrak X(D[S])\subseteq \{P[S]\mid P\in\mathfrak X(D)\}\cup \{\overline{Q}\cap D[S]\mid \overline{Q}\in\mathfrak X(K[S])\}=:\mathfrak P
\end{equation*}
and also
\begin{equation*}
\mathfrak X(D[S])\subseteq \{D[P]\mid P\in\mathfrak X(S)\}\cup \{\overline{Q}\cap D[S]\mid \overline{Q}\in\mathfrak X(D[G])\}=:\mathfrak P'.
\end{equation*}
\end{lemma}

\begin{proof}
To begin with, note that $P\in\spec(D)$ if and only if $P[S]\in\spec(D[S])$ \cite[Corollary 8.2]{Gil84}. Let $Q\in\mathfrak X(D[S])$, then $Q\cap D\in\spec(D)$. If $Q\cap D=P\neq (0)$, then $(0)\subsetneq P[S]\subseteq Q$, so by ht$(Q)=1$  equality holds. To see that $P\in\mathfrak X(D)$, just note that $(0)\subsetneq P'\subsetneq P$ implies $(0)\subsetneq P'[S]\subsetneq P[S]=Q$ for every $P'\in\spec(D)$. In case $Q\cap D=(0)$, we use the correspondence of prime ideals between $D[S]$ and $K[S]\cong (D\setminus\{0\})^{-1}D[S]$.\\
For the second inclusion, the proof is quite similar, but first we have to show the following\\
\textbf{Claim:} $P\in\spec(S)$ if and only if $D[P]\in\spec(D[S])$.
\begin{proof}[Proof of Claim]
If $D[P]\in\spec(D[S])$ then $P=D[P]\cap S$ is a prime ideal of $S$, so let $P\in\spec(S)$. Clearly, $D[P]$ is an ideal of $D[S]$. Let $f,g\in D[S]$ such that $fg\in D[P]$ and note that we can totally order $S$ in a way that is compatible with the monoid operation of $S$ \cite[Corollary 3.4]{Gil84}. Let $<$ denote such an order and write $f=\sum\limits_{i=1}^{n}a_iX^{s_i}$ and $g=\sum\limits_{j=1}^{m}b_jX^{t_j}$ with $s_1<\hdots <s_n$ and $t_1<\hdots <t_m$. If $f$ or $g$ are in $D[P]$ we are done, so assume that $f,g\notin D[P]$ and let $u\in[1,n]$ and $v\in[1,m]$ be minimal such that $s_u, t_v\notin P$. Then $s_u+t_v\notin P$ and the coefficient of $X^{s_u+t_v}$ in $fg$ is $a_ub_v\neq 0$; to wit: If not, then there exist $s_{u'}$ and $t_{v'}$ different from $s_u$ and $t_v$ such that $s_u+t_v=s_{u'}+t_{v'}\notin P$, hence both $s_{u'},t_{v'}\notin P$. Since $s_u$ and $t_v$ were chosen minimal with respect to not lying in $P$ we obtain $s_u<s_{u'}$ and $t_v<t_{v'}$ contradicting the equality $s_u+t_v=s_{u'}+t_{v'}$. Therefore $fg\notin D[P]$.
\qedhere[Proof of Claim]
\end{proof}

Now let $Q\in \mathfrak X(D[S])$, then clearly $Q\cap S\in\spec(S)$ and if $Q\cap S=P\neq \emptyset$ then $D[P]\subseteq Q$ is a non-zero prime ideal, thus by $Q\in\mathfrak X(D[S])$ equality holds. Assume to the contrary that ht$(P)\geq 2$. Then there is $P'\in\spec(S)$ such that $\emptyset\subsetneq P'\subsetneq P$, giving $(0)\subsetneq D[P']\subsetneq D[P]$ contradicting the fact that $D[P]\in\mathfrak X(D[S])$. For the case $Q\cap S=\emptyset$, just note that we have $D[G]=D[S]_N$, where $N=\{X^{\alpha}\mid \alpha\in S\}$ and we can again use the prime ideal correspondence.
\end{proof}

In fact, the above proof shows us, that the primes coming from $D[G]$ resp. $K[S]$ are always height-one again. Thus, the only primes that can produce a strict inclusion are height-one primes from $D$ resp. $S$ that do not induce height-one primes in the monoid algebra.

\begin{definition}\label{2.2}
Let $D$ be a domain and let $S$ be a torsion-free monoid. We say that
\begin{enumerate}
\item $D$ is $S$-UMT if $\mathfrak X(D[S])=\mathfrak P$ and
\item $S$ is $D$-UMT if $\mathfrak X(D[S])=\mathfrak P'$.
\end{enumerate}
\end{definition}

The next remark is just a consequence of the previous definition and remark.

\begin{remark}\label{2.3}
Let $D$ be a domain and let $S$ be a torsion-free monoid.
\begin{enumerate}
\item The following are equivalent:
	\begin{enumerate}
	\item $D$ is $S$-UMT.
	\item For all $P\in\mathfrak X(D)$ we have $P[S]\in\mathfrak X(D[S])$.
	\end{enumerate}
\item The following are equivalent:
	\begin{enumerate}
	\item $S$ is $D$-UMT.
	\item For all $P\in\mathfrak X(S)$ we have $D[P]\in\mathfrak X(D[S])$.
	\end{enumerate}
\end{enumerate}
\end{remark}

\begin{lemma}\label{2.4}
Let $D$ be a domain with quotient field $K$ and let $S$ be a torsion-free monoid with quotient group $G$. Then the following hold:
\begin{enumerate}
\item $D$ is $S$-UMT if and only if $D$ is $G$-UMT,
\item $S$ is $D$-UMT if and only if $S$ is $K$-UMT.
\end{enumerate}

\end{lemma}

\begin{proof}
1. $D[G]\cong D[S]_N$ where $N=\{X^{\alpha}\mid \alpha\in S\}$, so we have a correspondence of the height-one prime ideals of $D[S]$ not containing monomials and height-one prime ideals of $D[G]$. Thus for $P\in\mathfrak X(D)$ we have $P[S]\in\mathfrak X(D[S])$ if and only if $P[G]=P[S]_N\in\mathfrak X(D[G])$.\\
2. $K[S]\cong (D\setminus \{0\})^{-1}D[S]$, so we have a correspondence of the height-one prime ideals of $D[S]$ not containing non-zero constants and height-one prime ideals of $K[S]$. Thus for $P\in\mathfrak X(S)$ we have $D[P]\in\mathfrak X(D[S])$ if and only if $K[P]\in\mathfrak X(K[S])$.
\end{proof}

We now characterize the weakly Krull group algebras.

\begin{proposition}\label{2.5}
Let $D$ be a domain and let $G$ be a torsion-free abelian group that satisfies the ACC on cyclic subgroups. Then $D[G]$ is weakly Krull if and only if $D$ is weakly Krull and $G$-UMT.
\end{proposition}
\begin{proof}
"$\Rightarrow$" To prove that $D$ is $G$-UMT, it suffices by Lemma \ref{2.1} to show that $\mathfrak X(D[G])\supseteq \mathfrak P$, so let $P\in\mathfrak X(D)$ and $d\in P$. Since $D[G]$ is weakly Krull, by \cite[Theorem 3.1]{AMZ92} there exists a primary decomposition of $dD[G]=\bigcap_{i=1}^n A_i$, where $\sqrt{A_i}\in\mathfrak X(D[G])$ and since each $A_i$ contains $d$, we even have $\sqrt{A_i}=P_i[G]$ with $P_i\in\mathfrak X(D)$. Now $P[G]\supseteq \sqrt{dD[G]}=\sqrt{\bigcap_{i=1}^n A_i}=\bigcap_{i=1}^n\sqrt{A_i}=\bigcap_{i=1}^n P_i[G]$. Thus there is $i\in[1,n]$ such that $P_i[G]\subseteq P[G]$, hence $P_i\subseteq P$ and by $P\in\mathfrak X(D)$ equality holds.\\
Next we show that $D$ is weakly Krull. Let $P\in\mathfrak X(D)$, then $P[G]\in \mathfrak X(D[G])$, since we just proved that $D$ is $G$-UMT. Clearly, for all $x\in D$ we have $x\in P$ if and only if $x\in P[G]$, so $\mathfrak X(D)$ has finite character, since $\mathfrak X(D[G])$ has. Let $\frac{x}{y}\in\bigcap_{P\in\mathfrak X(D)}D_P$, then for all $P\in\mathfrak X(D)$ there are $x_P\in D$ and $y_P\in D\setminus P$ such that $\frac{x}{y}=\frac{x_P}{y_P}$. By what we just said, also $y_P\notin P[G]$, hence $\frac{x}{y}=\frac{x_P}{y_P}\in D[G]_{P[G]}$ for all $P\in\mathfrak X(D)$. Since $y\in D$, we also have $\frac{x}{y}\in D[G]_Q$ for all $Q\in\mathfrak X(D[G])$ not containing constants. Therefore in total we obtain that $\frac{x}{y}\in K\cap \bigcap_{Q\in\mathfrak X(D[G])}D[G]_Q=K\cap D[G]=D$.\\
"$\Leftarrow$" Let $D$ be weakly Krull and $G$-UMT. Obviously, $D[G]\subseteq \bigcap_{Q\in\mathfrak X(D[G])}D[G]_Q$, so it remains to prove the converse inclusion. For this, let $\frac{f}{g}\in\bigcap_{Q\in\mathfrak X(D[G])}D[G]_Q$. Then for all $Q\in\mathfrak X(D[G])$ there are $f_Q\in D[G]$ and $g_Q\in D[G]\setminus Q$ such that $\frac{f}{g}=\frac{f_Q}{g_Q}$, so $fg_Q=gf_Q$. We first show that $g\mid_{K[G]} f$: Since $K[G]$ is a factorial domain by \cite[Theorem 7.13]{GP74}, we have a prime factorization $g=p_1^{n_1}\cdots p_k^{n_k}$ in $K[G]$. Now $fg_Q=p_1^{n_1}\cdots p_k^{n_k}f_Q$ and if we choose $Q=p_iK[G]\cap D[G]\in\mathfrak X(D[G])$ then by $g_Q\notin Q$ we obtain $p_i^{n_i}\mid_{K[G]}f$. In total $g\mid_{K[G]} f$, hence without loss of generality we can write $\frac{f}{g}=\frac{f}{r}$ for appropriate $r\in D$ and obtain $fg_Q=rf_Q$.\\
Since $D$ is weakly Krull, we have a primary decomposition $rD=\bigcap_{i=1}^n I_i$ with $\sqrt{I_i}=P_i\in\mathfrak X(D)$. But now $fg_{P_i[G]}\in rD[G]=\bigcap_{j=1}^n I_j[G]$, in particular $fg_{P_i[G]}\in I_i[G]$. By \cite[Corollary 8.7]{Gil84}, $I_i[G]$ is a primary ideal with $\sqrt{I_i[G]}=\sqrt{I_i}[G]=P_i[G]$ and it follows that $f\in I_i[G]$ by $g_{P_i[G]}\notin P_i[G]$. Since we can do this for all $i\in[1,n]$ we obtain $f\in\bigcap_{i=1}^n I_i[G]=rD[G]$, hence $r\mid_{D[G]} f$ and $\frac{f}{r}\in D[G]$.\\
To see that the intersection is of finite character, just note that an $f\in D[G]$ can only be in finitely many height-one primes that come from $K[G]$ since it is a factorial domain and that it also can only be in finitely many height-one primes coming from $D$ since $D$ is weakly Krull.
\end{proof}

Next we characterize the weakly Krull monoid algebras over fields.

\begin{proposition}\label{2.6}
Let $K$ be a field and let $S$ be a torsion-free monoid with quotient group $G$ such that $G$ satisfies the ACC on cyclic subgroups. Then $K[S]$ is weakly Krull if and only if $S$ is weakly Krull and $K$-UMT. 
\end{proposition}
\begin{proof}
"$\Rightarrow$" Let $K[S]$ be weakly Krull and $P\in\mathfrak X(S)$. We have to show that $K[P]\in\mathfrak X(K[S])$. We already know by the claim in the proof of Lemma \ref{2.1} that $K[P]$ is a prime ideal, so it remains to prove its height is one. Let $\alpha\in P$, then we obtain a primary decomposition $X^{\alpha}K[S]=\bigcap_{i=1}^n A_i$ with $\sqrt{A_i}\in\mathfrak X(K[S])$. Moreover, since $X^{\alpha}\in A_i$ it holds that $\sqrt{A_i}=K[P_i]$ for $P_i\in\mathfrak X(S)$. It follows that $ K[P]\supseteq\sqrt{X^{\alpha}K[S]}=\sqrt{\bigcap_{i=1}^n A_i}=\bigcap_{i=1}^n \sqrt{A_i}=\bigcap_{i=1}^n K[P_i]$. Thus there is $i\in[1,n]$ such that $K[P_i]\subseteq K[P]$ and by $P\in\mathfrak X(S)$ equality follows, hence $K[P]\in\mathfrak X(K[S])$.\\
To prove that $S$ is weakly Krull, let $\alpha - \beta\in\bigcap_{P\in\mathfrak X(S)}S_P$. Then for all $P\in\mathfrak X(S)$ there exist $\alpha_P\in S$ and $\beta_P\in S\setminus P$ such that $\alpha - \beta=\alpha_P - \beta_P$. Since for all $a\in S$ we have $a\in P$ if and only if $a\in K[P]$, it follows that $X^{\alpha - \beta}=X^{\alpha_P - \beta_P}\in K[S]_{K[P]}$. Clearly, $\beta$ is in no height-one prime coming from $K[G]$, so in total we obtain $\alpha - \beta\in G\cap\bigcap_{P\in\mathfrak X(K[S])}K[S]_P=G\cap K[S]=S$. It is clear that $\mathfrak X(S)$ is of finite character, since $K[S]$ is weakly Krull.\\
"$\Leftarrow$" Clearly, $\mathfrak X(K[S])$ is of finite character, so it remains to prove that $K[S]=\bigcap_{Q\in\mathfrak X(K[S])}K[S]_Q$. We first prove the following\\
\textbf{Claim A:} For all $P\in\mathfrak X(S)$ we have $K[P]\in$ $t$-$\max(K[S])$.
\begin{proof}[Proof of Claim A]
Let $P\in\mathfrak X(S)$. To show that $K[P]$ is a maximal $t$-ideal, we prove that for every $f\in K[S]\setminus K[P]$ we have $(K[P], f)_t=K[S]$. Let $f\in K[S]\setminus K[P]$. Without loss of generality we can suppose that $f$ has no exponent in $P$, since if not, define $f'$ to be the part of $f$ containing all monomials not in $K[P]$, then $(K[P],f)=(K[P],f')$. Now $(K[S],f)_t=K[S]$ if and only if there exist $g_1,\hdots ,g_n\in K[P]$ such that $(g_1,\hdots ,g_n,f)_v=K[S]$ if and only if there exist $g_1,\hdots ,g_n\in K[P]$ such that $(g_1,\hdots ,g_n,f)^{-1}=K[S]$. Since $P$ is a maximal $t$-ideal of $S$, by the choice of $f=\sum_{i=1}^{l}a_iX^{s_i}$ (as always such that $s_1 <\hdots <s_l$) we have for all $i\in[1,l]$ that $(P,s_i)_t=S$, because $t$-$\dim(S)=1$. Thus for all $i\in[1,l]$ we can choose $b_1^{(i)},\hdots ,b_{n_i}^{(i)}\in P$ with $(b_1^{(i)},\hdots ,b_{n_i}^{(i)},s_i)^{-1}=S$. Set $\{b_1,\hdots , b_m\}=\{b_j^{(i)}\mid i\in[1,l], j\in[1,n_i]\}$. We claim that $(X^{b_1},\hdots ,X^{b_m},f)^{-1}=K[S]$. One inclusion is trivial, so let $h\in (X^{b_1},\hdots ,X^{b_m},f)^{-1}$ and note first, that $h\in K[G]$ since $hX^{b_1}\in K[S]$. Now write $h=\sum_{k=1}^{r}c_kX^{d_k}$ with $d_k\in G$ and $d_1<\hdots <d_r$ and we proceed by induction on $r$ to show that $h\in K[S]$: If $r=1$ this is clear, since for $h=X^{d}$ we have $d+s_1, d+b_1,\hdots, d+b_m\in K[S]$ hence in particular $d\in (b_1^{(1)},\hdots ,b_{n_1}^{(1)},s_1)^{-1}=S$. Now suppose that $r>1$, then by the same argument as for the case $r=1$ we obtain that $d_1\in S$ and thus $c_1X^{d_1}\in K[S]\subseteq (X^{b_1},\hdots ,X^{b_m},f)^{-1}$. Set $h_1=h-c_1X^{d_1}\in (X^{b_1},\hdots ,X^{b_m},f)^{-1}$. By induction hypothesis $h_1\in K[S]$ and therefore $h\in K[S]$.
\qedhere[Proof of Claim A]
\end{proof}
Using Claim A, we prove\\
\textbf{Claim B:} For all $\alpha\in S\setminus S^{\times}$ we have a primary decomposition of $(X^{\alpha})$ with associated primes of height one.
\begin{proof}[Proof of Claim B]
Let $\alpha\in S\setminus S^{\times}$. Then $X^{\alpha}$ is only contained in maximal $t$-ideals of the form $K[P]$ for $P\in\mathfrak X(S)$, since if $Q$ is a maximal $t$-ideal containing $X^{\alpha}$, then $Q\cap S$ is a non-empty prime $t$-ideal. One can see this as follows: Clearly, $K[Q\cap S]\subseteq Q$ and by \cite[Lemma 2.3.5]{BIK} we obtain $K[(Q\cap S)_t]=K[Q\cap S]_t\subseteq Q_t=Q$, hence $(Q\cap S)_t\subseteq 
Q\cap S$. In particular, if $Q$ is a maximal $t$-ideal, then $K[Q\cap S]\subseteq Q$ is a maximal $t$-ideal of $K[S]$, thus equality holds.\\
Since $S$ is weakly Krull, $\alpha$ is only contained in finitely many height-one primes of $S$, say $P_1,\hdots ,P_n$. It follows that the only maximal $t$-ideals of $K[S]$ containing $X^{\alpha}$ are the $K[P_1],\hdots ,K[P_n]$. Now
\begin{align*}
(X^{\alpha})=(\bigcap_{Q\in t\text{-}\max(K[S])} & K[S]_Q)X^{\alpha}=\bigcap_{Q\in t\text{-}\max(K[S])}K[S]_QX^{\alpha}=(\bigcap_{Q\in t\text{-}\max(K[S])}K[S]_QX^{\alpha})\cap K[S]=\\
& \bigcap_{Q\in t\text{-}\max(K[S])}(K[S]_QX^{\alpha}\cap K[S])=\bigcap_{i=1}^n (X^{\alpha}K[S]_{K[P_i]}\cap K[S]).
\end{align*}
It remains to prove, that the $X^{\alpha}K[S]_{K[P_i]}\cap K[S]$ are $K[P_i]$-primary, but this is clear, since the $K[P_i]$ are height-one primes.
\qedhere[Proof of Claim B]
\end{proof}
Now we take $\frac{f}{g}\in\bigcap_{Q\in\mathfrak X(K[S])}K[S]_Q$. Thus for every $Q\in\mathfrak X(K[S])$ there are $f_Q\in K[S]$ and $g_Q\in K[S]\setminus Q$ such that $\frac{f}{g}=\frac{f_Q}{g_Q}$, so $fg_Q=gf_Q$. We first show that $g\mid_{K[G]} f$: Since $K[G]$ is a factorial domain by \cite[Theorem 7.13]{GP74} we have a prime factorization $g=p_1^{n_1}\cdots p_k^{n_k}$ in $K[G]$. Now $fg_Q=p_1^{n_1}\cdots p_k^{n_k}f_Q$ and if we choose $Q=(p_i)\cap K[S]\in\mathfrak X(K[S])$ then by $g_Q\notin Q$ we obtain $p_i^{n_i}\mid_{K[G]}f$. In total $g\mid_{K[G]} f$, hence without loss of generality we write $\frac{f}{g}=\frac{f}{X^{\alpha}}$ for appropriate $\alpha\in S$ and obtain $fg_Q=X^{\alpha}f_Q$.\\
We show $f\in (X^{\alpha})$. By Claim B we pick a primary decomposition $(X^{\alpha})=\bigcap_{i=1}^n A_i$ with $\sqrt{A_i}=Q_i\in\mathfrak X(K[S])$. Let $i\in[1,n]$ and set $Q:=Q_i$. Then $fg_Q=X^{\alpha}f_Q\in (X^{\alpha})\subseteq A_i$, but $g_Q\notin Q=\sqrt{A_i}$, hence $f\in A_i$. Since $i$ was arbitrary, we obtain $f\in\bigcap_{i=1}^n A_i=(X^{\alpha})$.
\end{proof}

We are now in the position to prove our main result, characterizing the weakly Krull monoid algebras.

\begin{theorem}\label{2.7}
Let $S$ be a torsion-free monoid with quotient group $G$ such that $G$ satisfies the ACC on cyclic subgroups and let $D$ be a domain with quotient field $K$. Then $D[S]$ is weakly Krull if and only if $S$ is weakly Krull $K$-UMT and $D$ is weakly Krull $G$-UMT.
\end{theorem}
\begin{proof}
"$\Leftarrow$" If $D$ is weakly Krull $G$-UMT and $S$ is weakly Krull $K$-UMT, then $D[G]$ and $K[S]$ are weakly Krull by Proposition \ref{2.5} and Proposition \ref{2.6}. Thus $D[S]=D[G]\cap K[S]$ is weakly Krull.\\
"$\Rightarrow$" Let $D[S]$ be weakly Krull, then by \cite[Proposition 22.1]{HK98} both, $D[G]=D[S]_N$ (where $N=\{X^{\alpha}\mid \alpha\in S\}$) and $K[S]=(D\setminus \{0\})^{-1}D[S]$ are weakly Krull, hence $D$ is weakly Krull $G$-UMT by Proposition \ref{2.5} and $S$ is weakly Krull $K$-UMT by Proposition \ref{2.6}.
\end{proof}

Next (borrowing heavily from Chang \cite{Chang09}) we give an example of a group algebra over a field that is not weakly Krull, showing that the assumption of the group satisfying the ACC on cyclic subgroups cannot be omitted. Note that for a field $K$ and a group $G$ we always have that $K$ is weakly Krull $G$-UMT and that $G$ is weakly Krull $K$-UMT.

\begin{example}
Let $K$ be a field and let $\Q$ be the additive group of rational numbers. For each $n\in\N$ let $G_n$ be the subgroup of $\Q$ generated by $\frac{1}{2^n}$.
\begin{enumerate}
\item $K[G_n]$ is a principal ideal domain.
\item $1+X^{\frac{1}{2^n}}$ is a prime element of $K[G_n]$.
\item If $0\leq m<n$, then $1+X^{\frac{1}{2^m}}\notin (1+X^{\frac{1}{2^n}})K[G_n]$.
\item $t$-$\dim(K[\Q])=1$.
\item $K[\Q]$ is not weakly Krull.
\end{enumerate}
\end{example}
\begin{proof}
1. Since $G_n\cong \Z$ as additive groups, we have $K[G_n]\cong K[\Z]$, which is well known to be a PID.\\
2. Follows from the fact that $1+y$ is a prime element in $K[y]$ and that prime elements are lifted to prime elements under localization, provided they do not become units.\\
3. If we set $y=X^{\frac{1}{2^n}}$ and $k=2^{n-m}$, then $y$ is an indeterminate over $K$ and $K[G_n]\cong K[y,y^{-1}]$. Now $1+y^k\notin (1+y)K[y]$ since $k$ is even, thus $1+y^k\notin (1+y)K[y,y^{-1}]$, because $(1+y)K[y]$ is a prime ideal.\\
4. By \cite[Corollary 12.11.1]{Gil84} it suffices to prove that $G_n\subseteq \Q$ is a root extension, to prove that $K[G_n]\subseteq K[\Q]$ is an integral extension and hence $K[\Q]$ is of Krull dimension 1 by 1. and thus of $t$-dimension 1. To prove that we have a root extension, let $\frac{a}{b}\in\Q$ with $b\in\N$. Then $b\frac{a}{b}=a=(2^na)\frac{1}{2^n}\in G_n$.\\
5. We show that $1-X$ is in infinitely many height one prime ideals of $K[\Q]$. Clearly, $1-X=(1+X^{\frac{1}{2}})(1+X^{\frac{1}{2^2}})\cdots (1+X^{\frac{1}{2^n}})(1-X^{\frac{1}{2^n}})$, so it suffices to show that for $m\neq n$ the elements $1+X^{\frac{1}{2^n}}$ and $1+X^{\frac{1}{2^m}}$ are never in the same prime ideal, since by $t$-dimension is 1 they are contained in height-one prime ideals. Assume to the contrary that this was the case, say $P$ is a prime ideal containing $1+X^{\frac{1}{2^n}}$ and $1+X^{\frac{1}{2^m}}$ for $m<n$. Then $1+X^{\frac{1}{2^n}}\in P\cap K[G_n]$, hence $(1+X^{\frac{1}{2^n}})=P\cap K[G_n]$ by 1. and 2. Thus $1+X^{\frac{1}{2^m}}\in (1+X^{\frac{1}{2^n}})K[G_n]$, contradicting 3.
\end{proof}


\section{Applications}

In this section, we investigate the $K$-(resp. $G$-) UMT property for special monoids (resp. domains). Based on these results, we give applications of our main result Theorem \ref{2.7}. We start with non-negative monoids of totally ordered abelian groups, whose monoid algebras were recently studied in \cite{Gotti} and \cite{Gotti2}.

\begin{lemma}\label{2.8}
Let $(G,\leq)$ be a totally ordered abelian group satisfying the ACC on cyclic subgroups and $S=\{g\in G\mid g\geq 0\}$ be the non-negative monoid of $G$. Then $S$ is $K$-UMT for all fields $K$.
\end{lemma}
\begin{proof}
Assume to the contrary, that $S$ is not $K$-UMT for some field $K$. Then there is $P\in\mathfrak X(S)$ such that $K[P]\notin \mathfrak X(K[S])$, i.e. there is $A\in\mathfrak X(K[S])$ with $A\subsetneq K[P]$. Clearly, $A$ has no monomials, so has to be of the form $A=Q\cap K[S]$ for some $Q\in \mathfrak X(K[G])$. Since $K[G]$ is factorial, $Q=fK[G]$ for some prime $f\in K[G]$, hence the property $A\subseteq K[P]$ translates into "If $g\in K[S]$ such that $f\mid_{K[G]}g$, then $g\in K[P]$." Now let $f\in K[G]$ be as above and let $s_1$ be the smallest exponent of $f$. Then let $s\in G$ with $s+s_1=0$, thus $fX^{s}=:g\in K[S]$ and $f\mid_{K[G]}g$ but $g\notin K[P]$; a contradiction.
\end{proof}

\begin{corollary}
Let $D$ be a domain, $(G,\leq)$ be a totally ordered abelian group satisfying the ACC on cyclic subgroups and $S=\{g\in G\mid g\geq 0\}$ be the non-negative monoid of $G$. Then $D[S]$ is weakly Krull if and only if $D$ is weakly Krull $G$-UMT and $S$ is weakly Krull.
\end{corollary}

We are now able to prove that the classical UMT property is equivalent to our notion of being $\N_0$-UMT, provided the domain is weakly Krull. For this, we need the following preparatory

\begin{lemma}\label{2.9}
Let $D$ be a weakly Krull domain and $\N_0$-UMT. Then the polynomial ring $D[X]$ is weakly Krull.
\end{lemma}
\begin{proof}
Since $D$ is weakly Krull and $\N_0$-UMT it suffices to show that $\N_0$ is weakly Krull and $K$-UMT in order to apply Theorem \ref{2.7}, but that $\N_0$ is weakly Krull is well known (it is factorial) and the property of being $K$-UMT was shown in Lemma \ref{2.8}.
\end{proof}

\begin{proposition}\label{2.10}
Let $D$ be a weakly Krull domain, then the following are equivalent:
\begin{enumerate}
\item[(a)] $D$ is UMT,
\item[(b)] For all $n\in\N$, we have $D[X_1,\hdots ,X_n]$ is UMT,
\item[(c)] $D$ is $\N_0$-UMT,
\item[(d)] For all $n\in\N$, we have $D$ is $\N_0^n$-UMT,
\item[(e)] $D$ is $\Z$-UMT,
\item[(f)] For all $n\in\N$, we have $D$ is $\Z^n$-UMT.
\end{enumerate}
\end{proposition}

\begin{proof}
$(a)\Leftrightarrow (b)$ This is \cite[Theorem 2.4]{FGH98}.\\
$(a)\Leftrightarrow (c)$ $D$ weakly Krull UMT implies $D[X]$ weakly Krull by \cite[Proposition 4.11]{AHZ93}. Note that our proof of "$D[G]$ weakly Krull implies $D$ is $G$-UMT" in Proposition \ref{2.5} did not use the fact that $G$ was a group, so it shows $D$ is $\N_0$-UMT here. Conversely, if $D$ is weakly Krull and $\N_0$-UMT, then by Lemma \ref{2.9} $D[X]$ is weakly Krull, hence $D$ is UMT \cite[Proposition 4.11]{AHZ93}.\\
$(e)\Leftrightarrow (c)$ This follows from Lemma \ref{2.4}.\\
$(d)\Leftrightarrow (f)$ This follows from Lemma \ref{2.4}.\\
$(b)\Rightarrow (f)$ Let $n\in\N$, then since $D[\N_0]$ is weakly Krull UMT, $D[\N_0^n]$ is weakly Krull for all $n\geq 2$, hence $D[\Z^n]$ is by \cite[Proposition 22.1]{HK98}. Thus $D$ is $\Z^n$-UMT by Proposition \ref{2.5}.\\
$(d)\Rightarrow (c)$ This is trivial.
\end{proof}

Next we apply our main result to monoid algebras over affine monoids, which for example occur (over fields) when studying polytopes (see \cite{BG09}). Recall that an affine monoid is a torsion-free finitely generated monoid.

\begin{lemma}\label{2.11}
Let $S$ be an affine monoid. Then there is $n\in\N$ such that the quotient group $\mathsf q(S)\cong\Z^n$. In particular, $\mathsf q(S)$ satisfies the ACC on cyclic subgroups.
\end{lemma}
\begin{proof}
Since $S$ is a finitely generated submonoid of $\Z^m$ for some $m\in\N$, $\mathsf q(S)$ is a $\Z$-submodule of $\Z^m$ and therefore is itself a free $\Z$-module of rank $n\leq m$.
\end{proof}

The following remark is well known, but for the convenience of the reader and since the argument is short, we give it.

\begin{remark}\label{2.12}
Let $S$ be a finitely generated monoid, then its root closure $\widetilde{S}$ is Krull.
\end{remark}
\begin{proof}
Note that if $S$ is finitely generated, then also its reduced monoid $S_{red}$ is, so by \cite[Proposition 2.7.11]{Ge-HK06a} $\widetilde{S}=\widehat{S}$. By \cite[Theorem 2.7.13]{Ge-HK06a} we obtain that $S$ is $v$-noetherian and $(S:\widehat{S})\neq\emptyset$. Now the result follows from \cite[Theorem 2.3.5.3]{Ge-HK06a}.
\end{proof}

\begin{lemma}\label{2.13}
Affine monoids are $K$-UMT for all fields $K$.
\end{lemma}
\begin{proof}
Let $S$ be an affine monoid with quotient group $G$ and let $K$ be a field. In Lemma \ref{2.11} we just proved that $G$ satisfies the ACC on cyclic subgroups, thus in particular $S^{\times}$ does. Since $S$ is finitely generated, $\widetilde{S}$ is Krull by the above remark. It follows by \cite[Theorem 15.6]{Gil84} that $K[\widetilde{S}]$ is Krull. Since $K[\widetilde{S}]=\overline{K[S]}$ by \cite[Corollary 2.12.1]{Gil84} we obtain that $K[S]\subseteq K[\widetilde{S}]$ is an integral extension. Now we assume to the contrary that $S$ was not $K$-UMT, i.e. there is $P\in\mathfrak X(S)$ such that $K[P]\notin\mathfrak X(K[S])$. Since $K[P]$ cannot properly contain a prime ideal with monomials, it has to contain a prime ideal without monomials, hence a $Q$ coming from some $\overline{Q}\in\mathfrak X(K[G])$, so we have the situation $(0)\subsetneq Q\subsetneq K[P]$. Let $\widetilde{Q}=\overline{Q}\cap K[\widetilde{S}]$, then clearly $\widetilde{Q}\cap K[S]=Q$. Let $\widetilde{P}$ be the prime ideal of $\widetilde{S}$ corresponding to $P$ (\cite[Proposition 2.7]{Re13a}), then we prove $\widetilde{Q}\subseteq K[\widetilde{P}]$ obtaining a contradiction to the fact that $\widetilde{S}$ is $K$-UMT (because $K[\widetilde{S}]$ is weakly Krull).\\
Since $K[S]\subseteq K[\widetilde{S}]$ is an integral extension and $\widetilde{Q}\cap K[S]=Q$, by "going up" there exists $A\in\spec(K[\widetilde{S}])$ such that $\widetilde{Q}\subseteq A$ and $A\cap K[S]=K[P]$. Now $P\subseteq A\cap S$ implies $\widetilde{P}\subseteq A\cap \widetilde{S}$ and therefore we have $K[\widetilde{P}]\subseteq A$. Note that $\widetilde{P}\cap S=P$, thus $K[\widetilde{P}]\cap K[S]=K[P]$. Therefore $K[\widetilde{P}]\subseteq A$ and both are lying over $K[P]$. It follows by \cite[Corollary 5.9]{AtMa} that $A=K[\widetilde{P}]$.
\end{proof}

\begin{proposition}\label{2.14}
Let $S$ be an affine monoid and let $D$ be a domain. Then $D[S]$ is weakly Krull if and only if $D$ is weakly Krull UMT and $S$ is weakly Krull.
\end{proposition}
\begin{proof}
To begin with, note that by Lemma \ref{2.11} $G:=\mathsf q(S)\cong \Z^n$ satisfies the ACC on cyclic subgroups and that $D$ being UMT is equivalent to $D$ being $G$-UMT by Proposition \ref{2.10}. It remains to prove that $S$ is $D$-UMT, which is equivalent to $S$ being $K$-UMT by Lemma \ref{2.4}, where $K$ is the quotient field of $D$, but this is just Lemma \ref{2.13}. Now apply Theorem \ref{2.7}.
\end{proof}

Since fields are always UMT, an easy consequence of this proposition is the following

\begin{corollary}\label{2.15}
Let $K$ be a field and $S$ an affine monoid. Then $K[S]$ is weakly Krull if and only if $S$ is weakly Krull.
\end{corollary}

We now deduce the results by  Chouinard resp. El Baghdadi and Kim on Krull resp. generalized Krull monoid algebras from our main result Theorem \ref{2.7}.

\begin{proposition}
Let $D$ be a domain and $S$ a monoid with quotient group $G$. Then $D[S]$ is a Krull domain if and only if $D$ and $S$ are Krull and $G$ satisfies the ACC on cyclic subgroups.
\end{proposition}
\begin{proof}
"$\Rightarrow$" Let $K$ be the quotient field of $D$. Note that $K[G]$ is a Krull domain (as a localization of a Krull domain), thus by \cite[Lemma 1.2]{BK2016} $G$ satisfies the ACC on cyclic subgroups. Since saturated submonoids of Krull monoids are Krull again \cite[Theorem 2.4.8.1]{Ge-HK06a} and $D\setminus\{0\}$ and $S$ are saturated submonoids of $D[S]$, this implication follows.\\
"$\Leftarrow$"
First we show that $D$ is $S$-UMT.
Since $D$ is a Krull domain, for every $P\in\mathfrak X(D)$ there is a discrete rank-one valuation $\mathsf v_P:K\to \Z$, where $K$ is the quotient field of $D$, with valuation ring $D_P$. For $P\in\mathfrak X(D)$, we extend $\mathsf v_P$ to a discrete rank-one valuation on the quotient field of $D[S]$ via defining it on $D[S]$. Set
\begin{align*}
\mathsf v_{P[S]}:D[S]&\to \Z\\
\sum_{i=1}^n d_iX^{s_i}&\mapsto \inf\{\mathsf v_P(d_i)\mid i\in[1,n]\}.
\end{align*}
It is well-known that this defines a discrete rank-one valuation on the quotient field of $D[S]$ with valuation ring $D[S]_{P[S]}$ \cite[Theorem 15.3]{Gil84}. Therefore $P[S]$ is a height-one prime ideal of $D[S]$ and $D$ is $S$-UMT.\\
Next we prove that $S$ is $D$-UMT. Since $S$ is a Krull monoid, for every $P\in \mathfrak X(S)$ there is a discrete rank-one valuation $\mathsf w_P:G\to \Z$ with valuation monoid $S_P$. Again, for $P\in\mathfrak X(S)$ we extend $\mathsf w_P$ to a discrete rank-one valuation on the quotient field of $D[S]$ via defining it on $D[S]$. Set
\begin{align*}
\mathsf w_{D[P]}:D[S]&\to \Z\\
\sum_{i=1}^n d_iX^{s_i}&\mapsto \inf\{\mathsf w_P(s_i)\mid i\in[1,n]\}.
\end{align*}
It is well-known that this defines a discrete rank-one valuation on the quotient field of $D[S]$ with valuation ring $D[S]_{D[P]}$ \cite[Theorem 15.7]{Gil84}. Therefore $D[P]$ is a height-one prime ideal of $D[S]$ and $S$ is $D$-UMT.\\
Now, since Krull implies weakly Krull, we obtain by Theorem \ref{2.7} that $D[S]$ is weakly Krull. To prove that it is Krull, it suffices to prove that its localizations at height-one prime ideals are discrete rank-one valuation domains. For elements $A\in \mathfrak X(D[S])$, there are three cases:\\
\textbf{Case 1:} $A=D[P]$ for some $P\in\mathfrak X(S)$. For this case we just noted that $D[S]_{D[P]}$ is a discrete rank-one valuation domain.\\
\textbf{Case 2:} $A=P[S]$ for some $P\in\mathfrak X(D)$. Also for this case we noted above that $D[S]_{P[S]}$ is a discrete rank-one valuation domain.\\
\textbf{Case 3:} $A=qK[G]\cap D[S]$ for some prime element $q\in K[G]$. This case is clear, since $D[S]_A=K[G]_{qK[G]}$ and $qK[G]$ induces the standard discrete rank-one valuation.
\end{proof}

\begin{proposition}\label{Bag}
Let $D$ be a domain and $S$ a torsion-free monoid with quotient group $G$. Then $D[S]$ is a generalized Krull domain if and only if $D$ and $S$ are generalized Krull and $G$ satisfies the ACC on cyclic subgroups.
\end{proposition}
\begin{proof}
"$\Leftarrow$" This is the same proof as above, just delete the word "discrete" everywhere and substitute the value group $\Z$ by some appropriate height-one value group $\Gamma$.\\
"$\Rightarrow$" Let $D[S]$ be generalized Krull, equivalently $D[S]$ is weakly Krull PvMD \cite[Corollary 4.13]{AHZ93}. By \cite[Proposition 6.5]{AA} and Theorem \ref{2.7}, it follows that $D$ is a weakly Krull PvMD and $S$ is a weakly Krull PvMS (note that a PvMS is called a $t$-Prüfer monoid in \cite[Chapter 17]{HK98}), hence $D$ is generalized Krull and $S$ is generalized Krull by \cite[Theorem 17.2]{HK98}. For the statement on the the ACC on cyclic subgroups of $G$, note that $K[G]$ is generalized Krull and use \cite[Lemma 1.2]{BK2016}.
\end{proof}

Next we prove the $K$-UMT (resp. $G$-UMT) property for another class of monoids (resp. domains).

\begin{proposition}\label{2.16}
\begin{enumerate}
\item Let $D$ be a weakly Krull GCD-domain and $G$ be an abelian group that satisfies the ACC on cyclic subgroups. Then $D$ is $G$-UMT.
\item Let $K$ be a field and $S$ be a torsion-free weakly Krull GCD-monoid with quotient group $G$ such that $G$ satisfies the ACC on cyclic subgroups. Then $S$ is $K$-UMT.
\end{enumerate}
\end{proposition}
\begin{proof}
1. Since GCD-domains are PvMDs by \cite[Theorem 17.1]{HK98}, $D$ is a weakly Krull PvMD  and thus generalized Krull \cite[Corollary 4.13]{AHZ93}. Now Proposition \ref{Bag} implies that $D[G]$ is generalized Krull, hence weakly Krull, and so by Proposition \ref{2.5} $D$ is $G$-UMT.\\
2. If we can show that $S$ is generalized Krull we are done, since then $K[S]$ is generalized Krull by Proposition \ref{Bag}, hence weakly Krull, and so by Proposition \ref{2.6} $S$ is $K$-UMT. To see that $S$ is generalized Krull, we have to show that for all $P\in\mathfrak X(S)$ the $S_P$ are valuation monoids; the rest follows from the fact that $S$ is weakly Krull. The proof follows the same lines as Zafrullah's answer for the case of domains in \cite{mza}. Let $P\in\mathfrak X(S)$, then $S_P$ is a primary monoid and an easy calculation gives that it is also a GCD-monoid.\\
\textbf{Claim:} In a primary GCD-monoid $H$ no two non-units are coprime.
\begin{proof}[Proof of Claim]
Let $a,b\in H\setminus H^{\times}$ and assume to the contrary that $\gcd(a,b)=H^{\times}$. It follows by \cite[Proposition 10.2.5]{HK98} that $\gcd(a,b^m)=H^{\times}$ for all $m\in\N$. But since the monoid is primary, there is $n\in\N$ such that $a\mid b^n$, hence $a\in H^{\times}$, contradicting the choice of $a$.
\qedhere[Proof of Claim]
\end{proof}
To show that $S_P$ is a valuation monoid, let $x,y\in S_P$ with $d=\gcd(x,y)$. Then there are $a,b\in S_P$ such that $x=da$ and $y=db$ with $\gcd(a,b)=(S_P)^{\times}$. Now by our claim it follows without loss of generality that $a\in (S_P)^{\times}$, hence $a\mid b$ and therefore $x\mid y$.
\end{proof}

As an easy consequence of the previous statement, we reobtain Chang's result \cite{Chang09}.

\begin{corollary}\label{2.17}
Let $D$ be a domain and $S$ a torsion-free monoid whose quotient group $G$ satisfies the ACC on cyclic subgroups. Then $D[S]$ is a weakly factorial domain if and only if $D$ is a weakly factorial GCD-domain and $S$ is a weakly factorial GCD-monoid.
\end{corollary}
\begin{proof}
"$\Leftarrow$" By Propositon \ref{2.16} $D$ is $G$-UMT and $S$ is $K$-UMT, hence by Theorem \ref{2.7} $D[S]$ is weakly Krull and \cite[Theorem 14.5]{Gil84} implies that $D[S]$ is a GCD-domain and therefore is integrally closed. It follows that $Cl(D[S])\cong Cl(D)\oplus Cl(S)\cong 0$ \cite[Corollary 2.8]{BIK}, thus $D[S]$ is a weakly factorial GCD-domain.\\
"$\Rightarrow$" Let $D[S]$ be weakly factorial, equivalently $D[S]$ is weakly Krull with trivial $t$-class group. It follows by Theorem \ref{2.7} that $D$ and $S$ are weakly Krull and that they are GCD follows by a fact on splitting sets, see the proofs of \cite[Lemmata 7 and 8]{Chang09}. It remains to prove that their $t$-class groups are trivial. But it is well-known that their $t$-class groups embed into the $t$-class group of $D[S]$ (see for example \cite{BIK}).
\end{proof}

\providecommand{\bysame}{\leavevmode\hbox to3em{\hrulefill}\thinspace}
\providecommand{\MR}{\relax\ifhmode\unskip\space\fi MR }
\providecommand{\MRhref}[2]{%
  \href{http://www.ams.org/mathscinet-getitem?mr=#1}{#2}
}
\providecommand{\href}[2]{#2}


\begin{thebibliography}{1}


\bibitem{AA}
D.D. Anderson and D.F. Anderson, \emph{Divisorial ideals and invertible ideals in a graded integral domain}, J. Algebra \textbf{76} (1982), 549--569.

\bibitem{AHZ93}
D.D. Anderson, E.G. Houston and M. Zafrullah, \emph{t-linked extensions, the t-class group, and Nagata's theorem}, J. of Pure and Appl. Algebra \textbf{86} (1993), 109--124.


\bibitem{AMZ92}
D.D. Anderson, J.L. Mott and M. Zafrullah, \emph{Finite Character Representations for Integral Domains}, Bullettino U.M.I. (7) 6-B (1992), 613--630.

\bibitem{AtMa}
M.F. Atiyah and I.G. MacDonald, \emph{Introduction to Commutative Algebra}, Addison-Wesley, 1969.


\bibitem{BGGS}
N.R. Baeth, A. Geroldinger, D.J. Grynkiewicz and D. Smertnig, \emph{A semigroup theoretical view of direct-sum decompositions and associated combinatorial problems}, 
        Journal of Algebra and its Appl., \textbf{14:1550016}, (2015).


\bibitem{BG09}
 W. Bruns and J. Gubeladze, \emph{Polytopes, Rings, and K-Theory}, Springer, 2009.
 
 
\bibitem{Chang09}
G. W. Chang, \emph{Semigroup Rings as Weakly Factorial Domains}, Comm. Algebra \textbf{37:9} (2009), 3278--3287.

\bibitem{CHKK02}
S. T. Chapman, F. Halter-Koch and U. Krause, \emph{Inside factorial monoids and integral domains}, J. Algebra \textbf{252:2} (2002), 350--375.


\bibitem{Chou81}
L. Chouinard, \emph{Krull semigroups and divisor class groups}, Canad. J. Math. \textbf{33} (1981), 1459--1468.

\bibitem{Gotti}
J. Coykendall and F. Gotti, \emph{On the atomicity of monoid algebras}, J. Algebra, \textbf{539} (2019) 138-151.


\bibitem{BIK}
S. El Baghdadi, L. Izelgue and S. Kabbaj, \emph{On the class group of a graded domain}, J. of Pure and Appl. Algebra \textbf{171} (2002), 171--184.


\bibitem{BK2016}
S. El Baghdadi and H. Kim, \emph{Generalized Krull Semigroup Rings}, Comm. Algebra \textbf{44:4} (2016), 1783--1794.



\bibitem{FGH98}
M. Fontana, S. Gabelli and E. Houston, \emph{UMT-domains and domains with Pr\"ufer integral closure}, Comm. Algebra \textbf{26} (1998), 1017--1039.


\bibitem{Ge-HK06a}
A.~Geroldinger and F.~Halter-Koch, \emph{Non-{U}nique {F}actorizations.
  {A}lgebraic, {C}ombinatorial and {A}nalytic {T}heory}, Pure and Applied
  Mathematics, vol. 278, Chapman \& Hall/CRC, 2006.


\bibitem{Gil84}
R. Gilmer, \emph{Commutative semigroup rings}, Chicago Lectures in Mathematics, 1984.


\bibitem{Gil92}
R. Gilmer, \emph{Multiplicative Ideal Theory}, Queen's Papers in Pure and Applied Mathematics, vol. 90, 1992.


\bibitem{GP74}
R. Gilmer and T. Parker, \emph{Divisibility properties in semigroup rings}, Michigan Math. J. \textbf{21} (1974), 65--86.


\bibitem{Gotti2}
F. Gotti, \emph{Irreducibility and factorizations in monoid rings}, in: \emph{Numerical Semigroups} (2020), Eds. V. Barucci, S. T. Chapman, M. D'Anna, and R. Fröberg
Springer INdAM Series, Vol. 40, Switzerland, pp 129-139.


\bibitem{HK95}
F. Halter-Koch, \emph{Divisor theories with primary elements and weakly Krull domains}, Boll. Un. Mat. Ital. B \textbf{9} (1995), 417--441.


\bibitem{HK98}
F. Halter-Koch, \emph{IDEAL SYSTEMS: An Introduction to Multiplicative Ideal Theory}, Marcel Dekker, 1998.


 \bibitem{Re13a}
  A.~Reinhart, \emph{On integral domains that are $\rm{C}$-monoids}, Houston J.
    Math. \textbf{39} (2013), 1095 -- 1116.
    

\bibitem{mza}
mzafrullah (https://math.stackexchange.com/users/60902/mzafrullah), Local GCD domain and not Bezout domain, URL (version: 2013-02-04): https://math.stackexchange.com/q/294314

\end{thebibliography}
\end{document}